\documentclass[10pt, reqno]{amsart}

\usepackage{amsthm, amssymb, amsmath, enumerate}
\usepackage{verbatim}
\usepackage{graphicx, graphics}
\usepackage{algorithm}

\usepackage{tikz}
\usetikzlibrary{arrows}
\tikzstyle{block}=[draw opacity=0.7,line width=1.4cm]

\usepackage[all,arc,curve,frame,color]{xy}
\usepackage{subfigure}
\usepackage{url}

\newtheorem{thm}{Theorem}[section]
\newtheorem{lem}[thm]{Lemma}
\newtheorem{cor}[thm]{Corollary}
\newtheorem*{cor*}{Corollary}
\newtheorem{prop}[thm]{Proposition}

\newtheorem*{conjecture*}{Conjecture}

\theoremstyle{remark} 
\newtheorem*{question*}{Question}

\newtheorem{remark*}[thm]{Remark}
\newtheorem{example}[thm]{Example}

\theoremstyle{definition} 
\newtheorem{define}[thm]{Definition}
\newtheorem*{define*}{Definition}

\numberwithin{equation}{section}  

\newcommand{\ZZ}{\mathbb{Z}}     
\newcommand{\PP}{\mathbb{P}}      
\newcommand{\QQ}{\mathbb{Q}}      

\usepackage{xcolor}

\newcommand{\Hom}{\operatorname{Hom}}

\newcommand{\Gal}{\operatorname{Gal}}  
\newcommand{\an}[1]{\operatorname{an}}  
\newcommand{\Aut}{\operatorname{Aut}}   
\newcommand{\PGL}{\operatorname{PGL}}

\newcommand{\tth}{^{\operatorname{th}}}

\newcommand{\Manoa}{M\=anoa}
\newcommand{\Hawaii}{Hawai\kern.05em`\kern.05em\relax i}

\newcommand{\id}{\mathrm{id}}

\DeclareMathOperator{\PrePer}{PrePer}
\DeclareMathOperator{\Twist}{Twist}

\title{Uniform bounds for pre-periodic points in families of twists}

\author{Alon Levy}
\address{ICERM, 
Providence, RI}
\email{levy@math.brown.edu}

\author{Michelle Manes}
\author{Bianca Thompson}
\thanks{The second and third authors' work partially supported by NSF-DMS 1102858.}

\address{Department of Mathematics,
University of \Hawaii\ at \Manoa, 
Honolulu, HI}
\email{mmanes@math.hawaii.edu, bat7@hawaii.edu}

 \date{\today}

\begin{document}

\maketitle

	\begin{abstract}
Let $\phi$ be a morphism of $\PP^N$ defined over a number field $K.$ We prove that there is a bound $B$ depending only on $\phi$ such that every twist of $\phi$ has no more than $B$ $K$-rational preperiodic points.  (This result is analagous to a result of Silverman for abelian varieties~\cite{height}.)  For two specific families of quadratic rational maps over $\QQ$, we find the bound $B$ explicitly.    
\end{abstract}

\section{Introduction}\label{sec:intro}
Let $K$ be a number field and $\phi: \PP_K^N \rightarrow  \PP_K^N$ be a morphism.  We denote by $\phi^n$ the $n^{\mathrm{th}}$ iterate of $\phi$ under composition.  A point $P \in \PP^N$ is periodic if there exists an integer $n>0$ such that $\phi^n(P)=P$, and $P$ is preperiodic if there exist integers $n > m \geq 0$ such that $\phi^n(P) = \phi^m(P)$.  Let 
\[
\PrePer\left(\phi, \PP^N_K\right) = \{ P \in \PP_K^N \colon P \text{ is preperiodic under } \phi\}.
\]

A motivating problem in the field of arithmetic dynamics is the uniform boundedness conjecture of  Morton and Silverman~\cite{ratper}.  

\begin{conjecture*}\label{MScon}
Let $K/\QQ$ be a number field of degree $D$, and let $\phi:\PP^N \rightarrow \PP^N$ be a morphism of degree $d\geq 2$ defined over $K$.  There is a constant $\kappa(D, N, d)$ such that
$$\#\PrePer(\phi,\PP^N_K) \leq \kappa(D, N, d).$$
\end{conjecture*}

This is a deep and difficult problem.  It implies, for example, uniform boundedness for torsion points on abelian varieties over number fields (see~\cite{ub4av}).  Even the special case $N=1$ and $d=4$ is enough to imply Merel's uniform boundedness of torsion points on elliptic curves proved in~\cite{ecuniform}.   Though much work has been done on this problem for nearly~$20$ years, to date only non-uniform bounds are known.
As a first step, one might ask if Conjecture~\ref{MScon} holds in interesting families of dynamical systems.  For example, Poonen conjectures a precise bound for quadratic polynomials over~$\QQ$ in~\cite{poonenrefined}.

In the present work, we show that Conjecture~\ref{MScon} holds for families of twists of rational maps.  More precisely, for any twist $\psi$ of a rational map $\phi$ defined over a number field $K$, the number of $K$-rational preperiodic points is uniformly bounded by a constant depending only on the map $\phi$, but independent of the twist.  (See Section~\ref{sec:prelim} for relevant definitions.)  This involves finding a bound on the degree of the field of definition of a twist based on the size of the automorphism group of the map~$\phi$.

The statement of our main result in Theorem~\ref{cor:twist bound} and its proof are similar to a result for abelian varieties in~\cite{height}.  There, Silverman shows that given an abelian variety $A$, for all but finitely many twists $A_{\xi}$ of $A$ the set of $K$-rational torsion points $A_{\xi}(K)_{\text{tors}}$ is contained in 
\[
\{ P \in A_{\xi}(\bar K) \colon \text{for some } f \in \Aut{A_\xi}, f \neq \id, f(P) = P\}.
\]

\subsection*{Acknowledgments}  We would like to  thank the Institute for Computational and Experimental Research in Mathematics for a productive and enjoyable semester during which much of this project was completed.  We benefited greatly from productive conversations with  workshop participants, including Holly Krieger, William Gignac, Frank Paladino, Chi-hao Wang, Andrew Bridy, Kevin Doerksen, Jacqueline Anderson, Jan-Li Lin, Katherine Stange and especially Mike Zieve, who suggested the full strength of Theorem~\ref{cor:twist bound}. Thanks also to Joseph Silverman, Xander Faber, and Bianca Viray for comments on an earlier draft.

\section{Uniform Bounds for Families of Twists}\label{sec:prelim}

We begin with some background in arithmetic dynamics.  Throughout, $K$ will be a number field, and we will state explicitly when results hold for more general fields.   

\begin{define}
$\Hom^N_d(K)=\{\phi:\PP^N_K \rightarrow\PP^N_K \colon \phi \text{ is a morphism of degree } d^N\}.$ That is, $\phi$ is defined in each coordinate by homogeneous polynomials of degree $d$ with coefficients in $K$.
\end{define}

The primary tool in our proof is a classic theorem of Northcott~\cite{northcott}.

\begin{thm}[Northcott]\label{Northcott}
Let $\phi\in \Hom_d^N(K)$. Then the set of preperiodic points
$\PrePer(\phi, \bar K) $ is a set of bounded height. In particular,
$\PrePer(\phi, \PP^N_K) $
is a finite set, and more generally, for any $D > 1$ the set
\[
\bigcup_{[L:K]\leq D}
\PrePer(\phi, \PP^N_L)
\]
is finite.
 \end{thm}

Let $f \in \text{PGL}_{N+1}(\bar K)$ act on the points of $\PP^N_{\bar K}$ as a fractional linear transformation in the usual way.   Then we define the morphism 
$$\phi^f = f\circ \phi \circ f^{-1}.$$

\begin{define}
 We say that two morphisms $\phi$ and $\psi$ are \emph{conjugate} if there is some $f\in \PGL_{N+1}\left(\bar K\right)$ such that $\phi^f = \psi$.  They are \emph{conjugate over $K$} if there is some $f\in \PGL_{N+1}(K)$ such that $\phi^f = \psi$.  
 \end{define}
 
 \bigskip

If $P$ is a point of  period $n$ for $\phi$, then $f(P)$ has the same property for $\phi^f$, and similarly for preperiodic points.  It is also easily seen that $\left(\phi^n\right)^f = \left(\phi^f\right)^n$.  So  conjugate maps have essentially the same dynamical behavior.  However, if we are concerned with the arithmetic of the (pre)periodic points, we must be a bit more careful.
For a map $\phi \in \Hom^N_d(K)$, 
$$
\Twist(\phi/K) = \left\{ 
\begin{matrix}
\text{$K$-equivalence classes of maps $\psi \in \Hom^N_d(K)$ } \\
\text{such that $\psi$ is $\overline K$-equivalent to $\phi$}\\
\end{matrix} \right\}.
$$

\begin{example}
Let 
\[
\phi(z) = z - \frac 2 z \text{ and }\psi(z) = z - \frac 1 z.
\]
  Also let $f(z) = z\sqrt{2}$.  One may check that $\phi^f (z) = \psi(z)$.  So $\psi$ is a (quadratic) twist of $\phi$.  Solving
$\phi^2(z) = z$ gives the $\QQ$-rational two-cycle $\{\pm 1\}$.  But $\psi$ does not have rational points of period~$2$; solving $\psi^2(z) = z$ gives   $\{\pm 1/\sqrt 2\}$.
\end{example}

\begin{define}
           For any $\phi \in \Hom_d^N$ define $\mathcal{A}_\phi$ to
be the  \emph{automorphism group} of $\phi$, i.e.,
           \[
          \mathcal{A}_\phi = \{f \in \PGL_{N+1} \mid \phi^f = \phi\}.
          \]
          From~\cite{levy, petsche}, $\mathcal{A}_{\phi}$ is well-defined
as a finite subgroup of $\PGL_{N+1}$.
       \end{define}
 
We introduce some notation to make the statement and proof of Lemma~\ref{lem:bound deg} more succinct.  
Let $\phi\in \Hom_d^N(K)$,
 and  $\psi \in \Twist(\phi/K)$. 
 So there exists an $f\in \PGL_{N+1}$ 
 with $\phi^f=\psi$. 
 Write $f=(a_{ij})$, a matrix. At least one of the $a_{ij}$ is nonzero, say $a_{lm}\neq 0$.  Then $f=(a_{ij}')$ where $a_{ij}'=\frac{a_{ij}}{a_{lm}}$ represents the same element of $\PGL_{N+1}.$ 
 
 \begin{define}
 Let $K_f=K(a_{ij}')$ be the
 minimal field of definition for a given $f\in\PGL_{N+1}$ such that $\phi^f = \psi$, and let $L_f$ be the Galois closure of $K_f$.
 \end{define}

\begin{lem}\label{lem:bound deg}
For any $\psi \in \Twist(\phi / K)$ and any $f$ satisfying $\phi^f = \psi$, 
 \[ [K_f : K ] \leq \#\mathcal A_\phi.\]
\end{lem}
\begin{proof}
Choose $f \in \PGL_{N+1}$ such that $\phi^f = \psi$, and let $\sigma\in \Gal(L_f / K)$. Since $\psi$ is defined over $K$, 
\begin{align*}
f\phi f^{-1}
&=(f\phi f^{-1})^{\sigma} 
=f^{\sigma} \phi ^{\sigma}(f^{-1})^{\sigma}\\
&=f^{\sigma}\phi(f^{-1})^{\sigma}.\\
\phi &= f^{-1}f^{\sigma}\phi(f^{\sigma})^{-1}f.
\end{align*}
Hence $f^{-1}f^{\sigma}\in \mathcal A_\phi$.   Define the map 
\begin{align*}
\rho:\Gal(L_f /  K ) & \rightarrow \mathcal A_\phi\\
\sigma &\mapsto f^{-1}f^{\sigma}
\end{align*}
Suppose
\begin{align*}
\rho(\sigma)=\rho(\tau) &\Rightarrow f^{-1}f^{\sigma}=f^{-1}f^\tau\\
&\Rightarrow f^{\sigma}=f^{\tau} \\
&\Rightarrow f=f^{\tau\sigma^{-1}}.
\end{align*}
So $\tau\sigma^{-1} \in \Gal(L_f / K_f)$ since it fixes~$f$.   Clearly if $\tau_1\sigma^{-1} = \tau_2\sigma^{-1}$ as elements of
$\Gal(L_f / K_f)$, then they are equal as elements of $\Gal(L_f / K)$, and  $\tau_1 = \tau_2$.  Hence  $\rho$ is at most $[L_f : K_f]$-to-$1$, and we conclude that
\[
[L_f : K_f] [K_f: K ] = [L_f : K ] \leq [L_f : K_f] (\#\mathcal A_\phi),
\]
which gives the result since all of the extensions are finite.
\end{proof}

The following is now a straightforward consequence of Northcott's Theorem.
\begin{thm}\label{cor:twist bound}
Let $K$ be a number field and let $\phi\in \Hom_d^N(K)$.  Then there is a uniform bound $B_\phi$ such that for all $\psi\in \Twist(\phi/K)$,
\[
\# \PrePer(\psi, \PP^N_K) \leq B_\phi.
\]
\end{thm}

\begin{proof}
Given  $\psi \in \Twist(\phi/K)$ and $f\in \PGL_{N+1}$ such that $\phi^f =  \psi$, every $K$-rational periodic point for $\psi$ corresponds  to a $K_f$-rational periodic point for $\phi$.
From Lemma~\ref{lem:bound deg},   $[K_f : K ] \leq \#\mathcal A_\phi := D$.
By Northcott's Theorem, 
\[
\bigcup_{[K_f:K]\leq D}
\PrePer(\phi, \PP^N_{K_f})
\]
is finite; call the size of this set $B_\phi$.  Then  for every $\psi \in \Twist(\phi/K)$, 
\[
\#\PrePer(\psi, \PP_K^N) \leq B_\phi. \qedhere
\]
\end{proof}

\begin{remark*}
In fact Lemma ~\ref{lem:bound deg} holds over an arbitrary field $K$; hence Theorem ~\ref{cor:twist bound} also holds when $K$ is a  function field over a finite field. 
\end{remark*}

\begin{example}
Suppose that $\phi(z) : \PP^1_K \to \PP^1_K$ is a degree-$2$ morphism with a unique fixed point $P \in \bar K$.  Applying any Galois action to the equation $\phi(P) = P$ shows that any Galois conjugate of $P$ is also fixed by $\phi$.  Hence, $P \in \PP^1_K$.
Choose a change of coordinates $f \in \PGL_2( K)$ moving $P$ to infinity.  Then we may write
\[
\phi^f(z) = \frac{z^2+az+b}{z+a} \text{ for some } a,b \in  K \text{ with } b \neq 0.
\]
The critical points of $\phi^f$ are $a\pm \sqrt b$, so conjugate by $g \in \PGL_2(K)$ which fixes infinity and moves the critical points to $\pm \sqrt b$, and we see that $\phi$ is conjugate over $K$ to a map of the form $\phi_b(z)=z + \frac b z$.

 For any $b \in K^*$, $\phi_b$ is a quadratic twist of $\phi(z) = z+ \frac 1 z$ with conjugating map $f_b=z\sqrt{b}$.  So $f_b \in \PGL_2(L)$ with $[L:K] \leq 2$.  Let $Q \in \PP^1_K$ be a $K$-rational preperiodic point for $\phi_b$.  Then $f_b(Q) $ is an $L$-rational preperiodic point for $\phi$.

By Theorem~\ref{cor:twist bound}, 
$\#\PrePer(\phi_b, \PP^1_K) \leq B$ for some absolute bound $B$, independent of the parameter $b$.  In Section~\ref{sec:ExplicitBound}, we show that for this twisted family of quadratic rational maps when $K=\QQ$, the bound is  actually~$6$. 
\end{example}

\section{Main Tool: Dynatomic Polynomials}\label{sec: dyn poly}
For a morphism   $\phi : \PP_K^1 \to \PP_K^1$  (i.e. when the dimension  $N=1$), we may choose $F,G \in K[z]$ such that $\phi(z) = F(z)/G(z)$ with $\deg\phi = \max\{\deg F, \deg G\} $.  Hence we use the terms ``morphism'' and ``rational map'' interchangeably  when the map is on the projective line.

In this section, we will write rational maps using homogeneous coordinates:
\begin{align*}
\phi: \PP^1 &\to \PP^1\\
[X:Y] &\mapsto [F(X,Y): G(X,Y)],
\end{align*}
where $F$ and $G$ are homogeneous polynomials of the same degree with no common factor.
Then for $n>1$,
\[
\phi^n[X : Y] = [F_n(X,Y) : G_n(X,Y)],
\]
where $F_n$ and $G_n$ are given recursively by
\begin{equation*}
F_n(X,Y) = F_{n-1}(F(X,Y), G(X,Y)) \text{ and } G_n(X,Y) = G_{n-1}(F(X,Y), G(X,Y)) .
\label{eqn:recursion}
\end{equation*}

\begin{define} 
The $n$-period polynomial of $\phi$ is
 \[
 \Phi_{\phi,n}(X,Y)=YF_n(X,Y)-XG_n(X,Y).
 \]
  The \textit{$n^{th}$dynatomic polynomial of $\phi$} is the polynomial 
  \[
  \Phi_{\phi,n}^*(X,Y)=\prod_{d\mid n}{\Phi_{\phi,d}(X,Y)^{\mu(\frac{n}{d})}},
  \]
  where $\mu$ is the Mobius function. When the function $\phi$ is clear, we will suppress it from the notation, writing simply $\Phi_n^*$.
\end{define}
See~\cite[p.149]{ads} for a proof that $ \Phi_n^*(X,Y)$ is indeed a polynomial.  Clearly it is then a homogeneous polynomial in $X$ and $Y$.  Further, we let 
\[
\nu_2(n) = \sum_{d\mid n}\mu\left(\frac n d\right) 2^d = \text{ degree of } \Phi_n^* \text{ for a quadratic rational map.}
\]

By construction, all points of period $d \mid n$ are roots of the $n$-period polynomial~$\Phi_n$.  One might hope that the roots of $\Phi_n^*$ are the points of minimal period $n$ (eliminating as roots points with period $d < n$).  This isn't quite the case, but it is true that every point with minimal period $n$ is a root of $\Phi_n^*$, and that fact is enough for our purposes.  See~\cite[Chapter~4]{ads} for details about dynatomic polynomials and their properties.

\begin{lem}\label{lem:pos power}
The following products are positive powers of $k$ for $n>1:$
\begin{enumerate}[\textup($1$\textup)]
\item
$\displaystyle \prod_{d\mid n}\left( k^{2^d-d-1} \right)^{\mu(\frac n d)}$.
\item $\displaystyle \prod_{d\mid n}\left( k^{2^d-1} \right)^{\mu(\frac n d)}$.
\item $\displaystyle \prod_{d\mid n}\left( k^{2^{d-1}} \right)^{\mu(\frac n d)}$.
\item  $\displaystyle \prod_{d\mid n}\left( k^{b(d)} \right)^{\mu(\frac n d)}$ where $b(d)=\left \lceil \frac{2(2^{d-1}-1)}{3}\right\rceil.$
\end{enumerate}
\end{lem}

\begin{proof}
(1) Consider 
\begin{align*}
 \prod_{d\mid n}\left( k^{2^d-d-1} \right)^{\mu(\frac n d)} &= k^{\sum_{d\mid n}{\mu(\frac{n}{d})(2^d-d-1)}}\\
 &=k^{\sum_{d\mid n}{\mu(\frac{n}{d})(2^d-d)}}.
\end{align*}
The last step follows from the fact that when $n>1$ the $\sum_{d\mid n}{\mu\left(\frac{n}{d}\right)}=0.$
Recall that $\sum_{d\mid n}{\mu(\frac n d)d}=\varphi(n)$ where $\varphi(n)$ is the Euler totient function.  Also,
\begin{align}\label{eqn:2^n bound}
\sum_{d\mid n}{\mu\left(\frac n d\right)2^d}&\geq2^n-\sum_{d\mid n\text{ } d\neq n}{2^d}\\
&\geq 2^n-2^{n-1}=2^{n-1}.
\end{align}
Hence
\begin{align*}
\sum_{d\mid n}{\mu\left(\frac{n}{d}\right)(2^d-d)}&=\sum_{d\mid n}{\mu\left(\frac{n}{d}\right)2^d}-\sum_{d\mid n}{\mu\left(\frac{n}{d}\right)d}\\
&=\sum_{d\mid n}{\mu\left(\frac{n}{d}\right)2^d}-\varphi(n)\\
&\geq \sum_{d\mid n}{\mu\left(\frac{n}{d}\right)2^d}-n\\
&\geq 2^{n-1}-n.
\end{align*}
After taking the derivative of $2^{x-1}-x,$ we see that the fuction is increasing as long as $x>1.$ Hence $\sum_{d\mid n}{\mu(\frac{n}{d})(2^d-d)}>0.$

(2) follows immediately from (1).


(3) follows immediately from equation~\eqref{eqn:2^n bound}, replacing $d$ by $d-1$ and using the fact that $n>1$. 

(4)
Now consider 
\begin{align*}
\sum_{d\mid n}{\mu\left(\frac{n}{d}\right)b(d)}\geq \sum_{d\mid n}{\mu\left(\frac{n}{d}\right)\frac{2}{3}(2^{d-1}-1)}.
\end{align*} 
We are taking the result from (3) and multiplying it by $\frac 2 3,$ which is also positive.   
\end{proof}

\begin{lem}\label{lem:content}
Let $K$ be a number field and let 
\begin{align*}
\phi_{k,b}: \PP^1 &\to \PP^1\\
[X:Y] &\mapsto [k(X^2+bY^2): XY]
\end{align*}
 for some $b,k\in K^*$.  Then for every $n >  1$, 
 \begin{enumerate}[\textup($1$\textup)]
 \item
 The coefficient of  $X^{\nu_2(n)}$  in    $\Phi_{\phi_{k,b};n}^*(X,Y)$ is a positive power of $k$ times $C_n(k)$, where $C_n(k)$ refers to the $n\tth$ cyclotomic polynomial in the variable $k$.
 \item
The coefficient of  $Y^{\nu_2(n)}$  in  $\Phi_{\phi_{k,b};n}^*(X,Y)$ is a positive power of $k$ times a positive power of $b$.
\item
Each monomial of $\Phi_n^*$ is divisible by $k$.
 \end{enumerate}
  \end{lem}

\begin{proof}
We have 
\begin{align}
F_1(X,Y)&=k(X^2+bY^2)  & G_1(X,Y)&=XY\nonumber\\
F_n(X,Y)&=k(F_{n-1}^2+bG_{n-1}^2) & G_n(X,Y)&=F_{n-1}G_{n-1}.\label{eqn:double recursion}
\end{align} 
A simple induction argument shows that for $n>1$, we have $\deg_X(F_n) = \deg_X(G_n) + 1$, and in fact $\deg_X(F_n) = 2^n$ and $\deg_X(G_n) = 2^n-1$.   (The same arguments hold for $\deg_Y F_n$ and $\deg_Y G_n$.)

Now,
\[
\text{ coefficient of } X^{2^n} \text{ in } F_n = k \left(\text{coefficient of } X^{2^{n-1}} \text{ in } F_{n-1}\right)^2,
\]
so inductively this coefficient is $k^{2^n-1}$.  Similarly,
\begin{align*}
\text{coeff. of } X^{2^n-1} \text{ in } G_n 
&=  \left(\text{coeff. of } X^{2^{n-1}} \text{ in } F_{n-1}\right)\left(\text{coeff. of } X^{2^{n-1}-1} \text{ in } G_{n-1}\right)\\
&= Y \prod_{i=0}^{n-1}k^{2^i - 1} = Yk^{2^n-n-1}.
\end{align*}

Let $c_d$ be the coefficient of $X^{2^d}$ in $\Phi_d = YF_d - X G_d$ and $c_d^*$ be the coefficient of $X^{\nu_2(d)}$ in $\Phi_d^*$.  So that 
\begin{align*}
c_d & = Y\left( k^{2^d-1} - k^{2^d-d-1}\right) = Y k^{2^d-d-1}\left(k^d-1\right)\\
c_n^* &= \prod_{d\mid n} c_d^{\mu(\frac n d)} = \prod_{d\mid n} \left(Y k^{2^d-d-1}\left(k^d-1\right)\right)^{\mu(\frac n d)} \\
&= \prod_{d\mid n} Y^{\mu(\frac n d)}
 \prod_{d\mid n}\left( k^{2^d-d-1} \right)^{\mu(\frac n d)}
  \prod_{d\mid n} \left(k^d-1\right)^{\mu(\frac n d)}.
\end{align*}
(Here, we use the definition of $\Phi_n^*$ and the fact that we know it is a polynomial in $X$ and $Y$.)  When $n>1$, the first term is~$1$, the second is a positive power of $k$ by Lemma~\ref{lem:pos power}, and the third is $C_n(k)$ exactly as claimed.   

Also,
\begin{align*}
\text{ coefficient of } Y^{2^n+1} \text{ in } \Phi_n
& = \text{ coefficient of } Y^{2^n} \text{ in } F_n\\
& = k \left(\text{coefficient of } Y^{2^{n-1}} \text{ in } F_{n-1}\right)^2,
\end{align*}
so inductively this coefficient is $k^{2^n-1}b^{2^{n-1}}$.  So then
\[
\text{ coefficient of } Y^{\nu_2(n)} \text{ in } \Phi_n^*
= \prod_{d\mid n} \left(k^{2^d-1}b^{2^{d-1}}\right)^{\mu\left(\frac n d\right)},
\]
which is a positive power of $k$ times a positive power of $b$ by Lemma~\ref{lem:pos power}.

The proof for the final claim is similar, but the algebraic details are messier.  We sketch the main points here and leave the details for the reader.  Inductively one may show that 
\[
F_n(X,Y) = k^{a(n)} f_n(X,Y) \text{ and } G_n(X,Y) = k^{b(n)} g_n(X,Y),
\]
where 
\[
a(n) = \frac{2^n - (-1)^n}{3}, \quad b(n) =\left \lceil \frac{2(2^{n-1}-1)}{3}\right\rceil
\]
and $f_n, g_n \in \ZZ[k,b,X,Y]$.  So then 
\[
\Phi_n(X,Y) = k^{b(n)} \Psi_n(X,Y),
\] 
where $\Psi_n  \in \ZZ[k,b,X,Y]$.  Then exactly as above, it follows that a positive power of $k$ divides each dynatomic polynomial $\Phi_n^*$.
\end{proof}

Since $\Phi_n^*(X,Y)$ is homogeneous, we may dehomogenize in the usual way.  We will write $\Phi_n^*(z)$ for the dehomogenized dynatomic polynomial.  Note that the lead coefficient of $\Phi_1^*(z) = k-1 = C_1(k)$ and the constant term is $bk$. 

Mobius inversion gives $\prod_{d\mid n}\Phi_{\phi_{k,b},d}^*(z)  = \Phi_{\phi_{k,b},n} \in \ZZ[k,b,z] $.  In other words, $\Phi_{\phi_{k,b},n}(z)$ factors over $\QQ(k,b)$, so by
Gauss's Lemma  it factors over $\ZZ[k,b]$.   Lemma~\ref{lem:content} and the remark above show  that the polynomials $\Phi_{d,\phi_{k,b}}^*(z)$ in the product each have content a non-negative power of $k$, meaning that $\Phi^*_{n,\phi_{k,b}}(z) \in \ZZ[k,b,z]$.

\begin{lem}\label{lem:monomials}
 Let $n>1$.  Then each monomial of the $n\tth$ dynatomic polynomial  $\Phi_n^*(X,Y)$ has the form $c_iX^{2i}Y^{\nu_2(n)-2i}b^{\frac{\nu_2(n) - 2i}{2}}$ where $c_i,b\in\QQ$.
\end{lem}

\begin{proof}
We first  show that for all $n\geq 1$, 
\begin{align}
F_n(X,Y) &=
\sum_{i=0}^{2^{n-1}}c_iX^{2i}Y^{2^n-2i}b^{2^{n-1}-i},
\text{ and }\label{eqn:F_n form}\\
G_n(X,Y)
&=Y\sum_{j=1}^{2^{n-1}} d_jX^{2j-1}Y^{2^n-2j}b^{2^{n-1}-j},\label{eqn:G_n form}
\end{align}
where $c_i, d_j \in \ZZ[k]$.

From equation~\eqref{eqn:double recursion}, we see that $F_1$ and $G_1$ have the correct form. Assume $F_{n-1}$ and $G_{n-1}$ satisfy the claim above. 
 
 Consider
\[
F_n
 =k\left(  \sum_{i=0}^{2^{n-2}}c_iX^{2i}Y^{2^{n-1}-2i}b^{2^{n-2}-i}  \right)^2
+kb\left(Y\sum_{j=1}^{2^{n-2}} d_jX^{2j-1}Y^{2^{n-1}-2j}b^{2^{n-2}-j}\right)^2.
\]
If we look at the first term monomial by monomial we get 
\begin{align*}
\left(c_iX^{2i}Y^{2^{n-1}-2i}b^{2^{n-2}-i} \right)&
\left(c_jX^{2j}Y^{2^{n-1}-2j}b^{2^{n-2}-j}\right)\\
&=c_ic_jX^{2(i+j)}Y^{2^n-2(i+j)}b^{2^{n-1}-(i+j)}, 
\end{align*}
so each monomial has the correct form.  Now consider monomials from the second term:
\begin{align*}
kbY^2\left( d_iX^{2i-1}Y^{2^{n-1}-2i}b^{2^{n-2}-i}\right)&
\left(d_jX^{2j-1}Y^{2^{n-1}-2j}b^{2^{n-2}-j}\right)\\
&=kd_id_jX^{2(i+j-1)}Y^{2^n-2(i+j-1)}b^{2^{n-1}-(i+j-1)},
\end{align*}
which has the correct form.
This completes the proof for $F_n$, and the proof for $G_n$ is similar.

It follows immediately from equations~\eqref{eqn:F_n form} and~\eqref{eqn:G_n form} that
\[
\Phi_n(X,Y)=
YF_n - XG_n =
Y\sum_{i=0}^{2^n}e_iX^{2i}Y^{2^n-2i}b^{2^{n-1}-i}.
\]

We now compute the dynatomic polynomial:
\begin{align*}
\Phi_n^*(X,Y)&=
\prod_{d \mid n}
\left(Y \sum e_i X^{2i} Y^{2^d-2i} b^{2^{d-1}-i} \right)^{\mu\left(\frac{n}{d}\right)}\\
&=\prod_{d \mid n} Y^{\mu\left(\frac{n}{d}\right)}
\prod_{d \mid n}
\left(\sum
e_i X^{2i} Y^{2^d-2i} b^{2^{d-1}-i}\right)^{\mu\left(\frac n d \right)}\\
&=
\prod_{d \mid n}
\left(\sum
e_i X^{2i} Y^{2^d-2i} b^{2^{d-1}-i}\right)^{\mu\left(\frac n d \right)},
\end{align*}
where the last step follows because $n>1$.


Let
\[
\alpha = \sum{e_iX^{2i}Y^{D_{\alpha}-2i}b^{\frac{D_{\alpha}}{2}-i}}
\text{ and }
\beta = \sum{f_jX^{2j}Y^{D_{\beta}-2j}b^{\frac{D_{\beta}}{2}-j}}.
\]
Clearly the form is not affected if we add or subtract two such monomials. It is then easy to check that $\alpha\beta$ and $\frac{\alpha}{\beta}$ also have the correct form.
\end{proof}



\begin{lem}\label{lem:hompoly}
For all $n \geq 1$, there exists a homogeneous polynomial $\psi_n(w,b)\in \ZZ[w,b]$ such that $\psi_n(\frac{z^2}{b},b)=\Phi_n^*(z,b)$.
\end{lem}

\begin{proof}
From Lemma~\ref{lem:monomials}, when $n>1$ each monomial of $\Phi_n^*(z)$ has the form $c_iz^{2i}b^{\frac{\nu_2(n) - 2i}{2}}$.  A straightforward calculation shows that $\Phi_1^*(z)$ also has this form.

Now substitute $w = z^2$ to get
$\Phi_n^*(w)$ with monomials of the form $c_iw^{i}b^{\frac{\nu_2(n) }{2} -i}$, which is homogeneous in $w$ and $b$, with degree $\frac{\nu_2(n) }{2}$.
\end{proof}

\begin{define}Let $F(X,Y)$ be a homogeneous polynomial. We define $\ell(F)$
to be the leading coefficient of the dehomogenized polynomial $F(z,1)$ and $c(F)$ to be the constant term of $F(z,1)$. 
\end{define}



%
%
%


\begin{lem}\label{lem:lead_and_const}
Let 
\begin{align*}
\phi_b(X,Y) \colon \PP^1& \to \PP^1\\
(X,Y)& \mapsto (X^2+bY^2, XY).
\end{align*}
Then
\begin{displaymath}
   \ell(\Phi^*_n(X,Y)) =
     \begin{cases}
       p &   \text{if }n=p^e, e\geq 1 \\
       0 &   \text{if } n = 1 \\
       1 &   \text{otherwise,}
     \end{cases}
\end{displaymath} 
and  $c(\Phi_n^*)$ is a non-negative power of~$b$.
\end{lem}

\begin{proof}
From Lemma~\ref{lem:content} part~(1), $\ell(\Phi_n^*)$ is a non-negative power of $k$ times $C_n(k)$, where $C_n(k)$ is the $n\tth$ cyclotomic polynomial in the variable~$k$.  The result follows from evaluating this at $k=1$.  Similarly, the result for $c(\Phi_n^*)$ follows from part~(2) of Lemma~\ref{lem:content} and evaluating at $k=1$.  The result for $n=1$ follows from the remark after Lemma~\ref{lem:content} that $\ell(\Phi^*_n(X,Y)) = k-1$ in this case.
\end{proof}

\section{The Explicit Bound}\label{sec:ExplicitBound}
In this section, we find an explicit uniform bound  for the number of $\QQ$ preperiodic points for a one-parameter family of quadratic rational maps; namely, the quadratic maps with a unique fixed point.  The existence of such a bound follows immediately from Theorem~\ref{cor:twist bound}, so the content of this work is in finding the bound explicitly.
By construction, a rational point $(z_0, b_0)$ on the variety $V\left(\Phi_n^*(z,b)\right)$ corresponds to a quadratic rational map $\phi(z) = z+\frac {b_0} z$, and a rational point $z_0$ of period $n$ for $\phi$.  Note that $b_0 \neq 0$, since that value does not give a degree~$2$ rational map.

By Lemma~\ref{lem:hompoly}, we may substitute $w = z^2$ in $\Phi_n^*$, and the resulting polynomial $\psi_n(w,b) \in \ZZ[w,b]$ is homogeneous in $w$ and $b$.  So  if $V\left(\Phi_n^*(z,b)\right)$ has a  rational point $(z_0, b_0)$ with $b_0 \neq 0$, then $V\left(\psi_n(w,b)\right)$ has a  rational point $\left(\frac{z_0^2}{b_0}, b_0\right)$.    Since $\psi_n(w,b)$ is homogeneous, we may equivalently ask if $\psi_n(w,1)$ has a rational root.

\begin{thm}\label{thm:main_theorem}
The rational map $\phi(z)=z+\frac{b}{z}$ where $b\in\QQ$ has no rational points with exact period $n\geq5$.
\end{thm}
\begin{proof}
From Lemma~\ref{lem:lead_and_const}, we know that  for $n>1$ the lead coefficient of  $\psi_n(w,1) \in \ZZ[w]$ is either $1$ (if $n$ is not a prime power) or $p$ (if $n = p^e$) and the constant coefficient is ~$1$.  The only rational roots of such a polynomial are $\pm 1$ when $n$ is not a prime power,  and $\pm 1$ or  $\pm\frac{1}{p}$ when $n$ is a power of a prime.

In either case, we can have no more than four rational roots of $\psi_n$, which means no more than four rational points are on any given cycle, and this is independent of the parameter~$b$. 
\end{proof} 

We now combine Theorem~\ref{thm:main_theorem} with earlier work on the two-parameter family $\phi(z) = kz + \frac b z$ to find exactly what periods are possible for rational periodic points.

\begin{cor}\label{cor:realized_periods}
If $P \in \PP^1_\QQ$ is a periodic point for $\phi(z) = z + \frac b z$ where $b\in\QQ$, then $P$ is either the point at infinity or a point of period~$2$.
\end{cor}

\begin{proof}
From Theorem~\ref{thm:main_theorem}, we know that the period of $P$ must be less than~$5$.  

From~\cite[Theorem~4]{manespreper}, a map of the form $\phi(z) = kz + \frac b z$ with $k, b \in \QQ^*$  has a rational point of smallest period~$4$ if and only if there is some $m \in \QQ \smallsetminus \{ 0, \pm1 \}$ such that $k = 2m/(m^2-1)$ and $b = -m/(m^4-1)$.  However, there is no such $m$ with  $1 = 2m/(m^2-1)$.   We conclude that $P$ has period less than~$4$.

Similarly, \cite[Theorem~3]{manespreper} says that if $\phi(z) = kz + \frac b z$ with $k,b \in \QQ^*$, then $\phi(z)$ has no rational point of smallest period~$3$.  Hence, $P$ has period one or two.

By construction, $P$ is a fixed point if and only if $P$ is the point at infinity.  The only other possibility is a rational point of period~$2$.
\end{proof}

To finish finding the exact bound on the number of rational preperiodic points for a map of the form $\phi(z) = z + \frac b z$, we introduce a bit more notation.

\begin{define}
Let $m$ and $n$ be positive integers. Given a rational map $\phi$ and a point $P$ that is strictly preperiodic for $\phi$ (in other words, $P$ is preperiodic but not periodic), we say that $P$ has type $m_n$ if $P$ enters a cycle of exact length $m$ after $n$ iterations. That is, $\phi^{n+m}(P) = \phi^n(P)$, where $m\geq 1$ and $n>1$ are the smallest such integers.
\end{define}

\begin{cor}
Let $\phi(z) = z + \frac b z \in \QQ(z)$ with $b \neq 0$.  Then $\phi$ has either~$2$, $4$, or~$6$ rational preperiodic points.
\end{cor}

\begin{proof}
First note that for every nonzero $b \in \QQ$, the point at infinity is fixed, and $\phi(0)$ is infinity.  Hence, every  $\phi_b$ has a rational fixed point and a rational point of type~$1_1$.

Applying~\cite[Proposition~$6$]{manespreper}, we see that $\phi$ has a rational points of type $1_2$ if and only if $b = -c^2$ for some $c\in\QQ^*$.  However, all of these maps are conjugate over $\QQ$, so   take $b=-1$ as a representative.
From~\cite[Propositions~$7$ and~$8$]{manespreper}, we conclude that $\phi$ has no rational points of type $1_n$ for $n>2$.  

Applying~\cite[Proposition~$2$]{manespreper}, we see that $\phi$ has a rational point of period~$2$ if and only if $b = -2c^2$ for some $c\in\QQ^*$.  Again, all such maps are conjugate over $\QQ$, so we  take $b=-2$ as a representative.  It is a simple matter to check that in this case, $\phi_b$ has two rational points of type $2_1$.  (See Figure~\ref{fig:final_fig}.)
From~\cite[Proposition~$8$]{manespreper}, we conclude that $\phi_b$ has no rational points of type $2_n$ for $n>1$.  

Finally,~\cite[Proposition~$2$]{manespreper} says that $\phi_b$ has both rational points of type $1_2$ and rational points of period~$2$ if and only if we can solve $1 = 1/(x^2-1)$ with $x \in \QQ \smallsetminus \{0, \pm1 \}$.  Evidently, this is not possible.  
Hence for $b = -1$ (and all $\QQ$-conjugate maps), $\phi_b$ has four rational preperiodic points.  Similarly, for $b=-2$ (and all $\QQ$-conjugate maps), $\phi_b$ has six rational preperiodic points.
There are no $b$ values for which $\phi_b$ has more than six rational preperiodic points.
\end{proof}

We provide a graphical representation of all possible structures for rational preperiodic points for the family $\phi(z) = z + \frac b z$.  In these graphs, the vertices represent points in $\PP^1_{\QQ}$, and an arrow from vertex~$P$ to vertex~$Q$ indicates~$\phi(P)=Q$.

\begin{figure}[h]
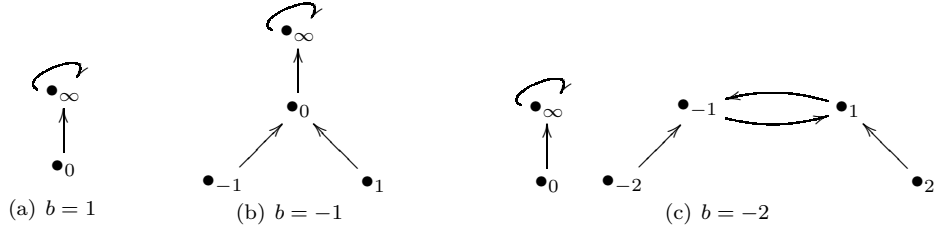

\centering
\mbox{
\subfigure[$b=1$]{
	\xygraph{ 
		!{<0cm,0cm>;<1cm,0cm>:<0cm,1cm>::} 
		!{(1,2) }*+{\bullet_{\infty}}="a" 
		!{(1,1) }*+{\bullet_{0}}="b" 
		"a" :@(l,ru) "a"
		"b":"a" 
	} 
}
\qquad\quad
\subfigure[$b=-1$]{
	\xygraph{ 
		!{<0cm,0cm>;<1cm,0cm>:<0cm,1cm>::} 
		!{(1,2) }*+{\bullet_{\infty}}="a" 
		!{(1,1) }*+{\bullet_{0}}="b" 
		!{(0,0) }*+{\bullet_{-1}}="c" 
		!{(2,0)}*+{\bullet_{1}}="d" 
		"a" :@(l,ru)  "a"
		"b":"a" 
		"d":"b" 
		"c":"b" 
	} 
}
}
\qquad\quad
\subfigure[$b=-2$]{
	\xygraph{ 
		!{<0cm,0cm>;<2cm,0cm>:<0cm,1cm>::} 
		!{(0,1) }*+{\bullet_{\infty}}="a" 
		!{(0,0) }*+{\bullet_{0}}="b" 
		 !{(.5,0) }*+{\bullet_{-2}}="c"
               !{(1,1) }*+{\bullet_{-1}}="d"
               !{(2.5,0) }*+{\bullet_{2}}="e"
               !{(2,1)}*+{\bullet_{1}}="f"
		"a" :@(l,ru)  "a" 
		"b":"a" 
		 "c":"d"
               "e":"f"
       "d":@/_/"f"
       "f":@/_/"d"
	}}
\caption{All possible rational preperiodic graphs for $\phi_b(z) = z + \frac b z$.}
\label{fig:final_fig}
\end{figure}

\section{Another Family of Twists}\label{sec:another fam}
We conclude with an abbreviated analysis of the possible preperiodic structures for another family of twists.   Lemma ~\ref{lem:content} says the lead coefficient of the dynatomic polynomials are powers of $k$ times a cyclotomic polynomial in $k$.  With the help of~\cite[Proposition 1]{Cyclotomic}, we can evaluate cyclotomic polynomials at roots of unity.  Hence, we consider $k=-1$.  The proofs are similar to those in Section~\ref{sec:ExplicitBound}, so the details will be sketched here.

We now consider the maps $\psi_b(z) = -\left(z + \frac b z\right)$.  Again, each $\psi_b$ is conjugate to $\psi_1$ via the map $f(z) = z \sqrt b$.  The family of twists is distinct from the one already considered, since $\psi_b$ has two finite fixed points at $\pm \sqrt{-b/2}$.

Note that Lemmas~\ref{lem:content}, \ref{lem:monomials}, and~\ref{lem:hompoly} apply to the family $\psi_b$, as we are taking $k=-1$.

\begin{lem}\label{lem:lead_and_const2}
Let 
\begin{align*}
\psi_b(X,Y) \colon \PP^1& \to \PP^1\\
(X,Y)& \mapsto \left(-(X^2+bY^2), XY\right).
\end{align*}
Then
\begin{displaymath}
   \ell(\Phi^*_n(X,Y)) =
     \begin{cases}
       \pm p &   \text{if }n=2p^e, e\geq 1, p\text{ prime} \\
       -2       &   \text{if }n = 1\\
       \pm 1 &   \text{otherwise,}
     \end{cases}
\end{displaymath} 
and  $c(\Phi_n^*)$ is a non-negative power of~$b$.
\end{lem}

\begin{proof}
The case $n=1$ is found by a simple computation.
By Lemma~\ref{lem:content}, the lead coefficient of $\psi_b$ is some power of $k$ times a cyclotomic polynomial. We apply~\cite[Proposition 1]{Cyclotomic} to evaluate the cyclotomic polynomials at $k=-1$, yielding the result. \end{proof}



\begin{prop}Let $\psi_b(z) = -\left(z + \frac b z\right) \in \QQ(z)$ with $b \neq 0$.  Then $\psi_b$ has either~$2$ or $4$ rational preperiodic points.
\end{prop} 
\begin{proof}
A proof identical to Theorem~\ref{thm:main_theorem} using Lemma~\ref{lem:lead_and_const2} shows there can be no rational points of period $n \geq 5$, and we apply~\cite[Theorems~$3$ and~$4$]{manespreper} to see that there are no rational points of primitive period~$3$ or~$4$.  From ~\cite[Proposition~$2$]{manespreper}, we conclude that no value of $b\in \QQ^*$ gives a rational map with a rational point of period~$2$.  Hence, we need only consider fixed points and points of type~$1_n$ for $n\geq 1$.

For every $b \in \QQ^*$, there is a rational fixed point at $\infty$ and a rational point of type $1_1$ at~$0$. By ~\cite[Proposition 1]{manespreper}, $\psi_b$ has finite fixed points if and only if $b=-2c^2$.  Since all such maps are conjugate over $\QQ$, so we take $b=-2$ as a representative.  There are no other type $1_1$ rational points. 
Applying ~\cite[Proposition 6]{manespreper} we have rational points of type $1_2$ if and only if $b=c^2$.  Again,  all such maps are conjugate over $\QQ$, so we take $b=1$ as a representative.
 By ~\cite[Proposition 7 and 8]{manespreper} there are no rational points of type $1_n$ for $n>2.$ 
 See Figure~\ref{fig:final_fig2}.
\end{proof}

\begin{figure}[h]
\centering
\mbox{
\subfigure[$b=2$]{
	\xygraph{ 
		!{<0cm,0cm>;<1cm,0cm>:<0cm,1cm>::} 
		!{(1,2) }*+{\bullet_{\infty}}="a" 
		!{(1,1) }*+{\bullet_{0}}="b" 
		"a" :@(l,ru) "a"
		"b":"a" 
	} 
}
\qquad\quad
\subfigure[$b=1$]{
	\xygraph{ 
		!{<0cm,0cm>;<1cm,0cm>:<0cm,1cm>::} 
		!{(1,2) }*+{\bullet_{\infty}}="a" 
		!{(1,1) }*+{\bullet_{0}}="b" 
		!{(0,0) }*+{\bullet_{-1}}="c" 
		!{(2,0)}*+{\bullet_{1}}="d" 
		"a" :@(l,ru)  "a"
		"b":"a" 
		"d":"b" 
		"c":"b" 
	} 
}
}
\qquad\quad
\subfigure[$b=-2$]{
	\xygraph{ 
		!{<0cm,0cm>;<2cm,0cm>:<0cm,1cm>::} 
		!{(0,1) }*+{\bullet_{\infty}}="a" 
		!{(0,0) }*+{\bullet_{0}}="b" 
		!{(.5,.5) }*+{\bullet_{-1}}="c" 
		!{(1,.5) }*+{\bullet_{1}}="d" 
		"a" :@(l,ru)  "a" 
		"b":"a" 
			"c" :@(l,ru)  "c" 
				"d" :@(l,ru)  "d" 
	}}
\caption{All possible rational preperiodic graphs for $\psi_b(z) =-\left(z + \frac b z\right)$.}
\label{fig:final_fig2}
\end{figure}

\end{document}